\documentclass[12pt]{amsart}
\usepackage{amscd,verbatim}
\usepackage{amssymb}
\usepackage[all]{xy}
\usepackage[
colorlinks,linkcolor=blue,citecolor=blue,urlcolor=red]{hyperref}
\usepackage{bm}

\newcommand{\A}{\mathbf{A}}

\newcommand{\Z}{\mathbb{Z}}

\newcommand{\sC}{\mathcal{C}}

\newcommand{\sY}{\mathcal{Y}}

\newcommand{\Cor}{\operatorname{\mathbf{Cor}}}

\newcommand{\ul}[1]{{\underline{#1}}}

\newcommand{\Set}{{\operatorname{\mathbf{Set}}}}
\newcommand{\PST}{{\operatorname{\mathbf{PST}}}}

\newcommand{\NST}{\operatorname{\mathbf{NST}}}

\newcommand{\DM}{\operatorname{\mathbf{DM}}}

\newcommand{\Sm}{\operatorname{\mathbf{Sm}}}
\newcommand{\Sch}{\operatorname{\mathbf{Sch}}}

\newcommand{\tr}{{\operatorname{tr}}}

\newcommand{\eff}{{\operatorname{eff}}}

\newcommand{\fin}{{\operatorname{fin}}}
\renewcommand{\o}{{\operatorname{o}}}

\newcommand{\red}{{\operatorname{red}}}

\newcommand{\id}{{\operatorname{Id}}}

\newcommand{\pr}{{\operatorname{pr}}}

\renewcommand{\lim}{\operatornamewithlimits{\varprojlim}}

\newcommand{\ol}{\overline}

\renewcommand{\phi}{\varphi}
\renewcommand{\epsilon}{\varepsilon}

\newcommand{\MNST}{\operatorname{\mathbf{MNST}}}

\newcommand{\MCor}{\operatorname{\mathbf{MCor}}}

\newcommand{\MSm}{\operatorname{\mathbf{MSm}}}
\newcommand{\ulMCorel}{\operatorname{\mathbf{\underline{M}Cor}^{\mathrm{el}}}}

\newcommand{\Bl}{{\mathbf{Bl}}}

\newcommand{\Sets}{\operatorname{\mathbf{Sets}}}

\newcommand{\Sq}{{\operatorname{\mathbf{Sq}}}}

\newcommand{\ulMSm}{\operatorname{\mathbf{\underline{M}Sm}}}

\newcommand{\ulMNST}{\operatorname{\mathbf{\underline{M}NST}}}
\newcommand{\ulMCor}{\operatorname{\mathbf{\underline{M}Cor}}}

\newcommand{\ulMEt}{\operatorname{\mathbf{\underline{M}Et}}}
\newcommand{\MEt}{\operatorname{\mathbf{MEt}}}

\newcommand{\OD}{\mathrm{OD}}
\newcommand{\ctimes}{\times^\mathrm{c}}

\newcommand{\Comp}{\operatorname{\mathbf{Comp}}}

\newcommand{\MV}{\operatorname{MV}}
\newcommand{\ulMV}{\operatorname{\underline{MV}}}
\newcommand{\ulMVfin}{\operatorname{\underline{MV}^{\mathrm{fin}}}}

\newcounter{spec}
{\end{list}}%

\newtheorem{Th}{Theorem}
\newtheorem{lemma}{Lemma}[subsection]
\newtheorem{thm}[lemma]{Theorem}
\newtheorem*{thm*}{Theorem}

\newtheorem{prop}[lemma]{Proposition}

\newtheorem{cor}[lemma]{Corollary}
\newtheorem{corollary}[lemma]{Corollary}

\theoremstyle{definition}
\newtheorem{defn}[lemma]{Definition}

\theoremstyle{remark}

\newtheorem{remark}[lemma]{Remark}

\newtheorem{claim}[lemma]{Claim}

\numberwithin{equation}{subsection}

\def\Comp{\Comp^{\fin}}

\def\MSm{\operatorname{\mathbf{MSm}}}
\def\MSm{\operatorname{\mathbf{MSm}}}

\def\ulMSm{\operatorname{\mathbf{\ul{M}Sm}}}

\def\ulMNST{\operatorname{\mathbf{\ul{M}NST}}}

\def\MNST{\operatorname{\mathbf{MNST}}}

\def\Comp{\operatorname{\mathbf{Comp}}}

\begin{document}

\title[Nisnevich topology with modulus]
{Nisnevich topology with modulus}
\author[H. Miyazaki]{Hiroyasu Miyazaki}
\address{RIKEN iTHEMS, Wako, Saitama 351-0198, Japan}
\email{hiroyasu.miyazaki@riken.jp}
\date{\today}
\thanks{This work is supported by RIKEN Special Postdoctoral Researchers (SPDR) Program, by RIKEN Interdisciplinary Theoretical and Mathematical Sciences Program (iTHEMS), and by JSPS KAKENHI Grant (19K23413).
}

\begin{abstract}
In the theory of motives \`a la Voevodsky, the Nisnevich topology on smooth schemes is used as an important building block. 
In this paper, we introduce a Grothendieck topology on proper modulus pairs, which will be used to construct a non-homotopy invariant generalization of motives.
We also prove that the topology satisfies similar properties to the Nisnevich topology.
\end{abstract}


\maketitle

\tableofcontents

\section{Introduction}\label{section:intro}
\setcounter{subsection}{1}

In the theory of motives \`a la Voevodsky in \cite{voetri}, the Nisnevich topology on the category of smooth schemes over a field $k$ plays a fundamental role. 
In this paper, we introduce a Grothendieck topology on proper modulus pairs, which will be used to construct a non-homotopy invariant generalization of motives.
We also prove that the topology satisfies similar properties to the Nisnevich topology.

A Nisnevich cover $f : Y \to X$ is an \'etale cover such that any point $x \in X$ admits a point $y \in Y$ with $f(y) = x$ and $k(y) = k(x)$. 
Therefore, the Nisnevich topology is finer than the Zariski topology and is coarser than the \'etale topology.
Voevodsky defined the category of effective motives $\DM^\eff$ as the derived category of the abelian category of Nisnevich sheaves with transfers $\NST$, modulo $\A^1$-homotopy invariance:
\begin{equation}\label{eq:DM}
\DM^\eff := \frac{\mathbf{D}(\NST )}{(\text{$\A^1$-homotopy invariance})}.
\end{equation}

We briefly recall the definition of $\NST$. 
Let $\PST$ be the category of additive abelian presheaves on the category of finite correspondences $\Cor$.
Then we have a natural functor $\Sm \to \Cor$, where $\Sm$ denotes the category of smooth schemes over $k$.
Then $\NST$ is defined to be the full subcategory of $\PST$ which consists of $F \in \PST$ such that the restriction $F|_{\Sm}$ is a Nisnevich sheaf on $\Sm$.

The definition of $\NST$ is simple, but it is non-trivial that $\NST$ is an abelian category. 
It follows from the existence of a left adjoint to the inclusion functor $\NST \to \PST$.
A key ingredient of the proof of its existence is the following fact: 
for any Nisnevich cover $U \to X$, the following \v{C}ech complex is exact as a complex of Nisnevich sheaves:
\begin{equation*}
\cdots \xrightarrow{} \Z_\tr (U \times_X U) \xrightarrow{} \Z_\tr (U) \xrightarrow{} \Z_\tr (X) \to 0,
\end{equation*}
where $\Z_\tr (-): \Cor \to \PST$ denotes the Yoneda embedding (see for example \cite[Prop. 6.12]{mvw}). 
Moreover, the Nisnevich topology is \emph{subcanonical}, i.e., every representable presheaf in $\Sm$ is a sheaf.

The category of motives $\DM^\eff$ has provided vast applications to the study of arithmetic geometry, but on the other hand, it has a fundamental constraint that it cannot capture {\it non-$\A^1$-homotopy invariant phenomena}, e.g., wild ramification.
Indeed, the arithmetic fundamental group $\pi_1 (X)$, which captures the information of ramifications, is not $\A^1$-homotopy invariant. 

An attempt to develop a theory of motives which captures non-$\A^1$-homotopy invariant phenomena started in \cite{motmod}.
The strategy is to extend Voevodsky's theory to \emph{modulus pairs}.
A \emph{modulus pair} is a pair $M=(\ol{M},M^\infty)$ of a scheme $\ol{M}$ and an effective Cartier divisor $M^\infty$ on $\ol{M}$ such that the \emph{interior} $M^\o := \ol{M} - M^\infty$ is smooth over $k$.
We can define a reasonable notion of morphisms between modulus pairs, and we obtain a category of modulus pairs $\ulMSm$.
A modulus pair $M$ is \emph{proper} if $\ol{M}$ is proper over $k$, and we denote by $\MSm$ the full subcategory of $\ulMSm$ consisting of proper modulus pairs
 (see Definition \ref{def:mod-pair} for details).

These categories embed in categories of ``modulus correspondences'' $\ulMCor$ and $\MCor$, just as $\Sm$ embeds in $\Cor$ (see Definition \ref{def:MCor}).
In \cite{motmod}, categories of ``modulus sheaves with transfers'' $\ulMNST$ (relative to $\ulMCor$) and $\MNST$ (relative to $\MCor$) were introduced, in order to parallel the definition of \eqref{eq:DM}. However, the proof that these categories are abelian was found to contain a gap.  This gap was filled in \cite{modsheaf1} for $\ulMNST$,  by showing that its objects are indeed the sheaves with transfers for a suitable Grothendieck topology on $\ulMSm$.

In this paper, we construct a Gro\-then\-dieck topology on $\MSm$ with nice properties. It will be shown in \cite{modsheaf2}, using \cite{cofinality}, that the objects of $\MNST$ are the sheaves (with transfers) for this topology and that this category is abelian. Thus the present paper contains the tools to finish filling the gap of \cite{motmod}. Moreover, we prove an important exactness result.

Our guide is the following characterization of the Nisnevich topology on $\Sm$:
the Nisnevich topology is generated by coverings $U \sqcup V \to X$ associated with some commutative square $S$ in $\Sm$ of the form 
\[\xymatrix{
W \ar[r] \ar[d] & V \ar[d] \\
U \ar[r] & X
}\]
which satisfies the following properties:
\begin{enumerate}
\item $S$ is a cartesian square,
\item the horizontal morphisms are open immersions,
\item the vertical morphisms are \'etale, and
\item the morphism $(V-W)_{\red} \to (X-U)_{\red}$ is an isomorphism.
\end{enumerate}

Such squares are called \emph{elementary Nisnevich squares}.
Elementary Nisnevich squares form a \emph{cd-structure} on $\Sm$ in the sense of \cite{cdstructures}.
A remarkable property of the Nisnevich cd-structure is the following fact:
a presheaf of sets $F$ on $\Sm$ is a Nisnevich sheaf if and only if $F(\emptyset) = \{\ast\}$ and for any elementary Nisnevich square as above, the square
\[\xymatrix{
F(X) \ar[r] \ar[d] & F(U) \ar[d] \\
F(V) \ar[r] & F(W)
}\]
is cartesian. 
This equivalence holds for any cd-structure which is complete and regular (see \cite[Def. 2.3, 2.10, Cor. 2.17]{cdstructures}). 

In \cite{modsheaf1}, a cd-structure on $\ulMSm$ is introduced. 
It is denoted $P_{\ulMV}$, and satisfies properties similar to elementary Nisnevich squares. Its definition will be recalled in \S 4.1. For short, we call the topology on $\ulMSm$ associated with $P_{\ulMV}$ the \emph{$\ulMV$-topology.}

\ 

Our main result is the following.

\begin{Th}
The category of proper modulus pairs $\MSm$ admits a cd-structure $P_{\MV}$ such that the following assertions hold:
For short, we call the topology associated with $P_{\MV}$ the $\MV$-topology. 
\begin{enumerate}
\item  (see Theorem \ref{thm:completeness}, \ref{thm:regularity}, \ref{thm:sheaf-criterion})
The cd-structure $P_{\MV}$ is complete and regular. 
In particular, a presheaf of sets $F$ on $\MSm$ is a sheaf for the $\MV$-topology if and only if $F(\emptyset ) = \{\ast \}$ and for any square $T \in P_{\MV}$ of the form 
\[\xymatrix{
W \ar[r] \ar[d] & V \ar[d] \\
U \ar[r] & M,
}\]
the square 
\[\xymatrix{
F(M) \ar[r] \ar[d] & F(U) \ar[d] \\
F(V) \ar[r] & F(W)
}\]
is cartesian.
\item  (see Theorem \ref{thm:subcanonicality})
The $\ulMV$-topology and the $\MV$-topology are subcanonical.
\item (see Corollary \ref{cor:Mayer-Vietoris})
For any $M \in \ulMSm$, consider the preshaf $\Z_\tr (M)$ on $\ulMCor$ represented by $M$, which is a sheaf for the $\ulMV$-topology by \cite[Th. 2 (2)]{modsheaf1}.
Then, for any square as above, the following complex of sheaves for the $\ulMV$-topology is exact:
\[
0 \to \Z_\tr (W) \xrightarrow{} \Z_\tr (U) \oplus  \Z_\tr (V)\xrightarrow{} \Z_\tr (M) \to 0 .
\]
\end{enumerate}
\end{Th}

The organization of the paper is as follows. 
In \S \ref{section:basics}, we recall basic definitions and results on modulus pairs from \cite{modsheaf1}.
In \S \ref{section-OD}, we introduce ``the off-diagonal functor'', which is a key ingredient to define the cd-structure on the category of proper modulus pairs. 
In \S \ref{section:cd-structure}, we define the cd-structure on the category of proper modulus pairs, and prove that it satisfies completeness and regularity. 
Finally, in \S \ref{section:Mayer-Vietoris}, we prove the exactness of the Mayer-Vietroris sequences associated with the distinguished squares with respect to the cd-structure.

\bigskip
\subsection*{Acknowledgements}
The author thanks Shuji Saito deeply for many helpful discussions on the first draft of the paper. 
It enabled the author to find a simple proof of Theorem \ref{thm:exactness}.
The author's gratitude also goes to Bruno Kahn who encouraged him to find a conceptual formulation of the cd-structure, which led to a considerable improvement of the paper. 
The author also thanks the referee for correcting some errors in the first version of the paper, and for providing helpful suggestions.

\bigskip
\subsection*{Notation and convention}
Throughout the paper, we fix a base field $k$. 
Let $\Sm$ be the category of separated smooth schemes of finite type over $k$,
and let $\Sch$ be the category of separated schemes of finite type over $k$.
For any scheme $X$ and for any closed subscheme $F \subset X$, we denote by $\Bl_F (X)$ the blow-up of $X$ along $F$.

\section{Basics on modulus pairs}\label{section:basics}

In this section, we introduce basic notions which we use throughout the paper.

\subsection{Category of modulus pairs}

We recall basic definitions on modulus pairs, introduced in \cite{modsheaf1}.
We will also introduce new notations.
Especially, \emph{the canonical model of fiber product} is often useful (see Definition \ref{def:canonical-model}). 
Though our main interest in this paper is on \emph{proper} modulus pairs, we introduce the general definition of modulus pairs for later use. 

\begin{defn}\label{def:mod-pair}
\ 

\begin{enumerate}
\item \label{def:mod-pair1}
A \emph{modulus pair} is a pair $M = (\ol{M},M^\infty)$ consisting of a scheme $\ol{M} \in \Sch$ and an effective Cartier divisor $M^\infty$ on $\ol{M}$ such that 
\begin{itemize}
\item \emph{the ambient space} $\ol{M} \in \Sch$, and
\item \emph{the modulus divisor} $M^\infty$, i.e., an effective Cartier divisor on $\ol{M}$
\end{itemize}
such that 
\begin{itemize}
\item \emph{the interior} $M^\o := \ol{M} \setminus |M^\infty |$ belongs to $\Sm$, where $|M^\infty|$ denotes the support of $M^\infty$.
\end{itemize}
Note that $M^\o$ is a dense open subset of $\ol{M}$. 
Moreover, we can prove that $\ol{M}$ must be a reduced scheme by using the smoothness of $M^\o$ and the assumption that $M^\infty$ is an effective Cartier divisor. 

\item \label{def:mod-pair2}
A modulus pair $M$ is called \emph{proper} if the ambient space $\ol{M}$ is proper over $k$.

\item \label{def:mod-pair3}
An \emph{admissible morphism} $f : M \to N$ of modulus pairs is a morphism between the interiors $f^\o : M^\o \to N^\o$ in $\Sm$ which satisfies 
\emph{the properness condition}:
\begin{itemize} 
\item Let $\Gamma$ be the graph of the rational map $\ol{f} : \ol{M} \dashrightarrow \ol{N}$ which is induced by $f^\o$. 
Then the natural morphism $\Gamma \to \ol{M}$ is proper.
\end{itemize}
and \emph{the modulus condition}:
\begin{itemize}
\item Let $\Gamma^N$ be the normalization of $\Gamma$. 
Then we have the following inequality 
\[
M^\infty |_{\Gamma^N} \geq N^\infty |_{\Gamma^N},
\]
of effective Cartier divisors on $\Gamma^N$, where $M^\infty |_{\Gamma^N}$ and $N^\infty |_{\Gamma^N}$ denote the pullbacks $M^\infty$ and $N^\infty$ along the natural morphisms $\Gamma^N \to \ol{M}$ and $\Gamma^N \to \ol{N}$. 
Note that the pullbacks are defined since the rational map $\ol{f}$ restricts to a morphism $f^\o$, and since $M^\o$ is dense in $\ol{M}$.
\end{itemize}
If $f : M \to N$ and $g : N \to L$ are admissible morphisms, then the composite $g^\o \circ f^\o : M^\o \to L^\o$ defines an admissible morphism $M \to N$ (cf. \cite{modsheaf1}).
If $N$ is proper, then the properness condition above is always satisfied. 

\item \label{def:mod-pair4}
We let $\ulMSm$ denote the category whose objects are modulus pairs and whose morphisms are admissible morphisms. 
The full subcategory of $\ulMSm$ consisting of proper modulus pairs is denoted by $\MSm$. 

\item \label{def:mod-pair5}
A morphism $f : M \to N$ in $\ulMSm$ is called \emph{ambient} if $f^\o : M^\o \to N^\o$ extends to a morphism $\ol{M} \to \ol{N}$ in $\Sch$. 
Such an extension is unique since $\ol{M}$ is reduced, $M^\o$ is dense in $\ol{M}$, and $\ol{N}$ is separated. 
We let $\ulMSm^\fin$ (resp. $\MSm^\fin$) denote the (non-full) subcategory of $\ulMSm$ (resp. $\MSm$) whose objects are modulus pairs (resp. proper modulus pairs) and whose morphisms are ambient morphisms. 

\item \label{def:mod-pair6}
A morphism $f : M \to N$ in $\ulMSm$ is called \emph{minimal} if $f$ is ambient and satisfies $M^\infty = \ol{f}^\ast N^\infty$.

\item \label{def:mod-pair7}
We let $\ul{\Sigma}_\fin$ denote the subcategory of $\ulMSm$ whose objects are the same as $\ulMSm$ and whose morphisms are those morphisms $f : M \to N$ in $\ulMSm^\fin$ such that $f$ is minimal, $\ol{f} : \ol{M} \to \ol{N}$ is proper and $f^\o : M^\o \to N^\o$ is an isomorphism in $\Sm$.
Then the canonical functor $\ulMSm^\fin \to \ulMSm$ induces an equivalence of categories $\ul{\Sigma}_\fin^{-1} \ulMSm^\fin \xrightarrow{\simeq} \ulMSm$ \cite[Prop. 1.9.2]{modsheaf1}.

\item \label{def:mod-pair8}
Let $\Sq$ be the product category $[0] \times [0]$, where $[0]= \{0 \to 1\}$.
For any category $\mathcal{C}$, we define $\mathcal{C}^\Sq$ to be the category of functors from $\Sq$ to $\mathcal{C}$.
An object $T$ of $\mathcal{C}^\Sq$ is given by a commutative diagram
\[\xymatrix{
T(00) \ar[r] \ar[d] & T(01) \ar[d] \\
T(10) \ar[r] & T(11).
}\]
in $\mathcal{C}$, and a morphism $T_1 \to T_2$ in $\mathcal{C}^\Sq$ is given by a set of morphisms $T_1 (ij) \to T_2 (ij)$, $i,j=0,1$, which are compatible with all the edges of the squares.

\item \label{def:mod-pair9}
A morphism $T_1 \to T_2$ in $\ulMSm^\Sq$ is called \emph{ambient} if for any $i,j=0,1$, the morphisms $T_1 (ij) \to T_2 (ij)$ in $\ulMSm$ are ambient.
A square $T \in \ulMSm^\Sq$ is called \emph{ambient} if it is contained in $(\ulMSm^\fin )^\Sq \subset \ulMSm^\Sq$.
\end{enumerate}
\end{defn}

The following lemma is often useful.

\begin{lemma}\label{lem:ambient-diagram}
For any square $T \in \ulMSm^\Sq$, there exists an ambient square $T'$ which admits an ambient morphism $T' \to T$ which is an isomorphism in $\ulMSm^\Sq$.
\end{lemma}

\begin{proof}
This is just a consequence of a repeated use of the graph trick (\cite[Lemma 1.3.6]{modsheaf1}).
Or the reader can consult the calculus of fractions in \cite[Prop. 1.9.2]{modsheaf1}). The details are left to the reader.
\end{proof}

\subsection{Fiber products}

We discuss fiber products in $\ulMSm$ and $\MSm$. 

\begin{lemma}\label{lem:univ-sup}
Let $X$ be a scheme, and let $D_1$ and $D_2$ be effective Cartier divisors on $X$.
Assume that the scheme-theoretic intersection $\inf (D_1 ,D_2) := D_1 \times_X D_2$ is also an effective Cartier divisor on $X$.
Set $X^\infty := D_1 + D_2 - \inf (D_1 ,D_2)$.

Then, for any morphism $f : Y \to X$ in $\Sch$ such that $Y$ is normal and the image of any irreducible component of $Y$ is not contained in $|X^\infty| = |D_1| \cup |D_2|$, we have 
\[
f^\ast X^\infty = \sup (f^\ast D_1 , f^\ast D_2),
\]
where $\sup$ is the supremum of Weil divisors on the normal scheme $Y$.
\end{lemma}

\begin{proof}
Since $\inf (D_1,D_2) \times_X Y = \inf (f^\ast D_1,f^\ast D_2)$, we are reduced to the case $X=Y$.
Moreover, an easy local computation shows that $D_1 - \inf (D_1,D_2)$ and $D_2 - \inf (D_1,D_2)$ do not intersect. The assertion immediately follows from this.
See \cite[Lem. 1.10.1, Def. 1.10.2 Rem. 1.10.3]{modsheaf1} for more details.
\end{proof}

\begin{defn}\label{def:canonical-model}
Let $f_1 : M_1 \to N$ and $f_2 : M_2 \to N$ be morphisms in $\ulMSm^\fin$, and
assume that the fiber product $P^\o := M_1^\o \times_{N^\o} M_2^\o$ exists in $\Sm$.
We define a modulus pair $P$ as follows. 
Let $\ol{P}_0$ be the scheme-theoretic closure of $P^\o$ in $\ol{M} \times_{\ol{N}} \ol{M}_2$,
and let $\ol{p}_{0,i} : \ol{P}_0 \to \ol{M}_1 \times_{\ol{N}} \ol{M}_2 \xrightarrow{\pr_i} \ol{M}_i$ be the composite of the closed immersion followed by the $i$-th projection for $i=1,2$.
Let
\[
\ol{P} := \Bl_{(\ol{p}_{0,1}^\ast M_1^\infty ) \times_{\ol{P}_0} (\ol{p}_{0,2}^\ast M_2^\infty )} (\ol{P}_0) \xrightarrow{\pi_P} \ol{P}_0
\]
be the blow-up of $\ol{P}_0$ along the closed subscheme $(\ol{p}_{0,1}^\ast M_1^\infty ) \times_{\ol{P}_0} (\ol{p}_{0,2}^\ast M_2^\infty )$. 
Set 
\[
P^\infty := \pi_P^\ast \ol{p}_{0,1}^\ast M_1^\infty + \pi_P^\ast \ol{p}_{0,2}^\ast M_2^\infty - E,
\]
where $E := \pi_P^{-1} ((\ol{p}_{0,1}^\ast M_1^\infty ) \times_{\ol{P}_0} (\ol{p}_{0,2}^\ast M_2^\infty ))$ denotes the exceptional divisor. 
Then we have $\ol{P} - |P^\infty | = P^\o \in \Sm$ by construction, and we obtain a modulus pair $P=(\ol{P},P^\infty)$.

We call $P$ \emph{the canonical model of fiber product of $f_1$ and $f_2$}, and we often write 
\[
M_1 \ctimes_N M_2 := P.
\] 
By construction, we have a commutative diagram 
\[\xymatrix{
M_1 \ctimes_N M_2 \ar[r]^{p_2} \ar[d]_{p_1} & M_2 \ar[d]^{f_2} \\
M_1 \ar[r]_{f_1} & N
}\]
in $\ulMSm^\fin$.
Moverover, we have $(M_1 \ctimes_N M_2)^\o \cong M_1^\o \times_{N^\o} M_2^\o$.
\end{defn}

\begin{thm}\label{thm:fiber-product-fin}
Let $f_1 : M_1 \to N$ and $f_2 : M_2 \to N$ be morphisms in $\ulMSm^\fin$.
Assume that the fiber product $M_1^\o \times_{N^\o} M_2^\o$ exists in $\Sm$.
Then the canonical model of fiber product $M_1 \ctimes_N M_2$ represents the fiber product $M_1 \times_N M_2$ in $\ulMSm$.
Moreover, if $M_1, M_2$, $N$ are proper, then $M_1 \ctimes_N M_2$ (hence $M_1 \times_N M_2$) is proper. 
\end{thm}

\begin{remark}
$M_1 \ctimes_N M_2$ does not necessarily represent a fiber product in $\ulMSm^\fin$, and it is not functorial in $\ulMSm^\fin$.
However, under some minimality conditions, they behave nicely in $\ulMSm^\fin$.
\end{remark}

\begin{proof}
We prove that $P:=M_1 \ctimes_N M_2$ satisfies the universal property of fiber product in $\ulMSm$. 
Let $g_1 : L \to M_1$ and $g_2 : L \to M_2$ be morphisms in $\ulMSm$ which coincide at $N$.
Since $\ulMSm \cong \ul{\Sigma}_\fin^{-1} \ulMSm^\fin$, we can find morphisms $L_1 \to L$ in $\ul{\Sigma}_\fin$ such that the composite morphisms $L_1 \to L \to M_i$ are ambient for $i=1,2$, and such that $\ol{L}_1$ is normal. 
Since $L_1 \to L$ is an isomorphism in $\ulMSm$, we replace $L$ with $L_1$ and assume that $\ol{L}$ is normal, and that $g_1$ and $g_2$ are ambient. 
Let $p_1 : P \to M_1$ and $p_2 : P \to M_2$ be the ambient morphisms as in Def. \ref{def:canonical-model}.

There exists a unique morphism $g^\o : L^\o \to P^\o = M_1^\o \times_{N^\o} M_2^\o$ in $\Sm$ which is compatible with $g_1^\o$, $g_2^\o$, $p_1^\o$ and $p_2^\o$.
It suffices to prove that $g^\o$ defines a morphism $L \to P$ in $\ulMSm$.
Let $\Gamma \subset \ol{L} \times \ol{P}$ be the closure of the graph of $g^\o$, and let $\Gamma^N$ be the normalization of $\Gamma$.
Let $s : \Gamma^N \to \ol{L}$ and $t : \Gamma^N \to \ol{P}$ be the natural projections.

Then, for $i=1,2$, we obtain a commutative diagram 
\[\xymatrix{
\Gamma^N \ar[r]^{t} \ar[d]_{s} & \ol{P} \ar[d]^{\ol{p}_i} \\
\ol{L} \ar[r]_{\ol{g}_i} \ar@{.>}^{g^\o}[ru] & \ol{M}_i
}\]
where the commutativity follows from the fact that $\ol{p}_i t$ and $\ol{g}_i s$ coincide on the dense open subset $s^{-1} (L^\o) \subset \Gamma^N$.

By the construction of $P$ and by Lemma \ref{lem:univ-sup}, we have
\begin{align*}
t^\ast P^\infty 
&= \sup (t^\ast \ol{p}_1^\ast M_1^\infty , t^\ast \ol{p}_2^\ast M_2^\infty ) \\
&= \sup (s^\ast \ol{g}_1^\ast M_1^\infty , s^\ast \ol{g}_1^\ast M_2^\infty ),
\end{align*}
where the second equality follows from the commutativity of the above diagram. 
Since $g_1$ and $g_2$ are ambient and $\ol{L}$ is normal, we have $\ol{g}_i^\ast M_i^\infty \leq L^\infty$. 
Therefore, we obtain 
\[
t^\ast P^\infty \leq s^\ast L^\infty,
\]
which shows that $g^\o$ defines a morphism $g : L \to P$. 
This proves the first assertion. 
The last assertion is obvious by construction. 
This finishes the proof. 
\end{proof}

\begin{cor}\label{thm:fiber-product}
Let $f_1 : M_1 \to N$ and $f_2 : M_2 \to N$ be morphisms in $\ulMSm$.
Assume that the fiber product $M_1^\o \times_{N^\o} M_2^\o$ exists in $\Sm$.
Then there exists a fiber product $M_1 \times_N M_2$ in $\ulMSm$.
Moreover, if $M_1$, $M_2$, and $N$ are proper, then $M_1 \times_N M_2$ is proper. 
\end{cor}

\begin{proof}
By \cite[Lemma 1.3.6]{modsheaf1}, for each $i=1,2$, there exists a morphism $M'_i \to M_i$ in $\ulMSm^\fin$ which is invertible in $\ulMSm$ and such that the composite $M'_i \to M_i \to N$ is ambient.
Theorem \ref{thm:fiber-product-fin} shows that the fiber product $M'_1 \times_N M'_2$ exists in $\ulMSm$. 
This also represents a fiber product $M_1 \times_N M_2$, proving the first assertion. 
The second assertion follows from the construction of the canonical model of fiber product. 
This finishes the proof.
\end{proof}

\begin{remark}\label{rem:tau-fb}
The inclusion functor $\tau_s : \MSm \to \ulMSm$ preserves fiber products by construction. 
\end{remark}

Given some minimality assumptions, we can say more about the canonical model of fiber product.
We will not need this in this paper, but it will be used in the other papers, including \cite{cofinality}.

\begin{prop}
\ 
\begin{enumerate}
\item  Let $f_1 : M_1 \to N$ and $f_2 : M_2 \to N$ be morphisms in $\ulMSm^\fin$, and assume that $f_1$ is minimal, $M_1^\o \times_{N^\o} M_2^\o$ is smooth over $k$ and $\ol{M}_1 \times_{\ol{N}} M_2^\infty$ is an effective Cartier divisor on $\ol{M}_1 \times_{\ol{N}} \ol{M}_2$.
Then we have 
\[
M_1 \ctimes_{N} M_2 = (\ol{M}_1 \times_{\ol{N}} \ol{M}_2 , \ol{M}_1 \times_{\ol{N}} M_2^\infty ).
\]

\item Consider the following commutative diagram 
\[\xymatrix{
U_1 \ar[r] \ar[d]_{j_1} & V \ar[d]_j & \ar[l] \ar[d]^{j_2} U_2 \\
M_1 \ar[r] & N & \ar[l] M_2
}\]
in $\ulMSm^\fin$, such that $j_1$ and $j_2$ are minimal, and such that $M_1^\o \times_{N^\o} M_2^\o$ and $U_1^\o \times_{V^\o} U_2^\o$ are smooth over $k$.
Then the morphism 
\[
j_1 \times j_2 : U_1 \ctimes_V U_2 \to M_1 \ctimes_N M_2
\]
 in $\ulMSm$, induced by the universal property of fiber product, belongs to $\ulMSm^\fin$ and is minimal. 

\item In the situation of (2), if $\ol{j}, \ol{j}_1,\ol{j}_2$ are open immersions, if $U_1 \to V$ is minimal and if $\ol{U}_1 \times_{\ol{V}} \ol{U}_2$ is normal, then 
\[
\ol{j_1 \times j_2} : \ol{U}_1 \times_{\ol{V}} \ol{U}_2 = \ol{U_1 \ctimes_V U_2} \to \ol{M_1 \ctimes_N M_2}
\]
is an open immersion, where the equality follows by (1).
\end{enumerate}
\end{prop}

\begin{proof}
(1): This follows from the construction of canonical model of fiber product (see also \cite[Corollary 1.10.7]{modsheaf1}). 

(2): Let $\ol{P}$ (resp. $\ol{Q}$) be the closure of $M_1^\o \times_{N^\o} M_2^\o$ (resp. $U_1^\o \times_{V^\o} U_2^\o$) in $\ol{M}_1 \times_{\ol{N}} \ol{M}_2$ (resp. $\ol{U}_1 \times_{\ol{V}} \ol{U}_2$).
Then the morphisms $\ol{j}_1$ and $\ol{j}_2$ induce a morphism \[\ol{J} : \ol{Q} \to \ol{P}.\]
Then we obtain the following commutative diagrams:
\[\xymatrix{
\ol{Q} \ar[r]^{\ol{J}} \ar[d]_{q_i} & \ol{P} \ar[d]^{p_i} \\
\ol{U}_i \ar[r]^{\ol{j}_i} & \ol{M}_i ,
}\]
in $\Sch$ for $i=1,2$, where $p_i$ and $q_i$ are the natural $i$-th projections.
Set $F := p_1^\ast M_1^\infty \times_{\ol{P}} p_2^\ast M_2^\infty \subset \ol{P}$ and $G := q_1^\ast U_1^\infty \times_{\ol{Q}} p_2^\ast U_2^\infty \subset \ol{Q}$.
Then the commutativity of the diagrams shows
\begin{align*}
\ol{J}^{-1} F &:= F \times_{\ol{P}} \ol{Q} = (q_i^\ast \ol{j}_1^\ast M_1^\infty ) \times_{\ol{Q}} (q_i^\ast \ol{j}_2^\ast M_2^\infty ) \\
&= q_i^\ast U_1^\infty  \times_{\ol{Q}} q_i^\ast U_2^\infty \\
&= G,
\end{align*}
where the equality in the second line follows from the minimality of $j_1$ and $j_2$.
Let $\pi_P : \Bl_{F} (\ol{P}) \to \ol{P}$ and $\pi_Q : \Bl_{G} (\ol{Q}) \to \ol{Q}$ be the  blow-ups.
Then, by the universal property of blow-up, $\ol{J}$ lifts to a morphism 
\[
\ol{J}_1 : \ol{U_1 \ctimes_V U_2} = \Bl_{G} (\ol{Q}) \to \Bl_{F} (\ol{P}) = \ol{M_1 \ctimes_N M_2},
\]
which makes the diagram 
\[\xymatrix{
\Bl_{G} (\ol{Q}) \ar[d]_{\pi_Q} \ar[r]^{\ol{J}_1} & \Bl_{F} (\ol{P}) \ar[d]^{\pi_P} \\
\ol{Q} \ar[r]^{\ol{J}} & \ol{P}
}\]
commute. 
Moreover, letting $F':=\pi_P^{-1} (F), G':=\pi_Q^{-1} (G)$ be the exceptional divisors, the commutativity of the two diagrams as above shows
\begin{align*}
\ol{J}_1^\ast (M_1 \ctimes_N M_2)^\infty &= \ol{J}_1^\ast (\pi_P^\ast p_1^\ast M_1^\infty + \pi_P^\ast p_1^\ast M_1^\infty - F') \\
&= \pi_Q^\ast \ol{J}^\ast p_1^\ast M_1^\infty + p_Q^\ast \ol{J}^\ast \pi_2^\ast M_2^\infty - G' \\
&= \pi_Q^\ast q_1^\ast \ol{j}_1^\ast M_1^\infty + \pi_Q^\ast q_2^\ast \ol{j}_2^\ast M_2^\infty - G' \\
&= \pi_Q^\ast q_1^\ast U_1^\infty + \pi_Q^\ast q_2^\ast U_2^\infty - G' \\
&= (U_1 \ctimes_V U_2)^\infty
\end{align*}
where the equality in the fourth line follows from the minimality of $j_1$ and $j_2$.
Therefore, the morphism $\ol{J}_1$ defines a minimal morphism $U_1 \ctimes_V U_2 \to M_1 \ctimes_N M_2$, as desired. 

(3): We take the notation as above. 
Then $\ol{U}_1 \times_{\ol{V}} \ol{U}_2$ is an open subset of $\subset \ol{P}$.
Since $\ol{J}^\ast F = G$, the minimality of $U_1 \to V$ shows $F \cap \ol{U}_1 \times_{\ol{V}} \ol{U}_2 = \ol{U}_1 \times_{\ol{V}} U_2^\infty$, where the right hand side is an effective Cartier divisor on $\ol{U}_1 \times_{\ol{V}} \ol{U}_2$.
Therefore, the blow-up $\pi_P$ is an isomorphism over $\ol{U}_1 \times_{\ol{V}} \ol{U}_2$, and the open immersion $\ol{U}_1 \times_{\ol{V}} \ol{U}_2 \to \ol{P}$ uniquely lifts to an open immersion $\ol{U}_1 \times_{\ol{V}} \ol{U}_2 \to \Bl_F (\ol{P})$. This finishes the proof.
\end{proof}

\subsection{A remark on elementary correspondences}

In this subsection, we will observe a relationship between cartesian squares and elementary correspondences. 
First we provide some definitions.

\begin{defn}\label{def:elem-set}
For any $M_1,M_2 \in \ulMSm$, we define $\ulMCorel$ to be the set of elementary finite correspondence $V : M_1^\o \to M_2^\o$ which satisfies the following \emph{admissibility conditions}: let $\ol{V}$ be the closure of $V$ in $\ol{M}_1 \times \ol{M}_2$, and let $\ol{V}^N \to \ol{V}$ be the normalization of $\ol{V}$.
Let $\pr_i:\ol{V}^N \to \ol{M}_i$ be the $i$-th projections. 
\begin{enumerate}
\item $\mathrm{pr}_1$ is proper.
\item $\pr_1^\ast M_1^\infty \geq \pr_2^\ast M_2^\infty$.
\end{enumerate}
\end{defn}

\begin{defn}[{\protect\cite[Def. 1.1.1, 1.3.3]{modsheaf1}}]\label{def:MCor}
A category $\ulMCor$ is defined as follows: the objects are the same as $\ulMSm$, and and for $M,N \in \ulMCor$, the set of morphisms is defined as the free abelian group generated on $\ulMCorel (M,N)$.
Note that $\ulMCor (M,N) \subset \Cor (M^\o,N^\o)$ by definition. 
The composition is given by the composition of finite correspondences. 
Define $\MCor$ as the full subcategory of $\ulMCor$ whose objects are proper modulus pairs. 
\end{defn}

\begin{prop}\label{prop:functor-elem}
For any modulus pair $M$, for any $f : N \to L$ in $\ulMSm$ and for any $V \in \ulMCorel (M,N)$, the image 
\[ f_+ (V):=(\id_{M^\o} \times f^\o )(V) \subset M^\o \times L^\o \]
 is an irreducible closed subset, and we have $f_+ (V) \in \ulMCorel (M,L)$.
 
Thus, any modulus pair $M$ is associated a covariant functor 
\[
\ulMCorel (M,-) : \ulMSm \to \Set .
\]
\end{prop}

\begin{proof}
By \cite[Prop. 1.2.3]{modsheaf1}, the composition of finite correspondences $W := \Gamma_{f^\o} \circ V$ belongs to $\ulMCor (M,L)$, where $\Gamma_{f^\o}$ denotes the graph of $f^\o : M^\o \to N^\o$. 
By the definition of composition, we can verify that $|W| = f_+ (V)$.
This implies that $f_+ (V)$ is a component of $W$.
Therefore, we have $W \in \ulMCor (M,L)$, as desired. 
\end{proof}

\begin{prop}\label{prop:fb-fc}
Let $T$ be a pull-back square in $\ulMSm$ of the form
\begin{equation}\label{eq:square-T}\begin{gathered}\xymatrix{
T(00) \ar[r]^{v_T} \ar[d]_{q_T} & T(01) \ar[d]^{p_T} \\
T(10) \ar[r]^{u_T} & T(11)
}\end{gathered}\end{equation}
and let $M$ be a modulus pair. 
Consider the associated commutative diagram of sets 
\[\xymatrix{
\ulMCorel (M,T(00)) \ar[r]^{v_{T+}} \ar[d]_{q_{T+}} & \ulMCorel (M,T(01)) \ar[d]^{p_{T+}} \\
\ulMCorel (M,T(10)) \ar[r]^{u_{T+}} & \ulMCorel (M,T(11)),
}\]
and set
\[
\Pi := \ulMCorel (M,T(10)) \times_{\ulMCorel (M,T(11))} \ulMCorel (M,T(01)).
\]
Then the induced map $\rho : \ulMCorel (M,T(00)) \to \Pi$ is surjective. 
Moreover, it is bijective if $v_T^\o$ is an immersion. 
\end{prop}

\begin{remark}
We can formulate another statement by replacing $\ulMCorel$ with $\ulMCor$ and $(-)_+$ with $(-)_\ast$, but it will be false.
Indeed, if $\alpha_1$ and $\alpha_2$ are distinct elementary correspondences which have the same image $\beta$ under $p_{T\ast}$, then the image of the (non-elementary) finite correspondence $\alpha := \alpha_1 - \alpha_2$ is zero, which is trivially contained in the image of $u_{T\ast}$. But there is no reason why $\alpha$ is contained in the image of $v_{T\ast}$.
\end{remark}

\begin{proof}
The latter statement is clear, since the composite $\mathrm{pr}_2 \circ \rho$ is equal to $v_{T+}$, which is injective if $v_T^\o$ is an immersion. 

We prove the surjectivity of $\rho$.
Take any $\alpha_1 \in \ulMCorel (M,T(10))$ and $\alpha_2 \in \ulMCorel (M,T(01))$, and assume $\beta := u_{T+} (\alpha_1) = p_{T+} (\alpha_2)$.
Let $\xi_i$ be the generic point of $\alpha_i$ for $i=1,2$.

We need to prove that there exists an element $\gamma \in \ulMCorel (M,T(00))$ which maps to $\alpha_1$ and $\alpha_2$.

Let 
\begin{align*}
\zeta 
&\in (M^\o \times T(10)^\o) \times_{M^\o \times T^\o (11)} (M^\o \times T(01)^\o) \\
&\cong M^\o \times T(10)^\o \times_{T^\o (11)} T(01)^\o \\
&\cong M^\o \times T(00)^\o.
\end{align*}
 be a point which lies over $\xi_1$ and $\xi_2$.
Let $\gamma := \ol{\{\zeta\}}$ be the closure of $\zeta$ in $M^\o \times T(00)^\o$, endowed with the reduced scheme structure.

\begin{claim}
$\gamma$ is an elementary correspondence from $M^\o$ to $T(00)^\o$.
\end{claim}

\begin{proof}
We have to prove that $\gamma$ is finite and surjective over a component of $M^\o$.
Since $\zeta = (\xi_1,\xi_2) \in \alpha_1 \times_{M^\o} \alpha_2$, the scheme $\gamma$ is naturally a closed subscheme of $\alpha_1 \times_{M^\o} \alpha_2$. 
Moreover, since $\zeta$ maps to $\xi_i$ via the projection $\pr_i : \alpha_1 \times_{M^\o} \alpha_2 \to \alpha_i$ for each $i=1,2$, we obtain dominant maps $\gamma \to \alpha_i$. 
These maps are finite (hence surjective) since each $\alpha_i$ is finite over $M^\o$.
Since the natural map $\gamma \to M^\o$ factors as $\gamma \to \alpha_1 \to M^\o$, and since $\alpha_1$ is finite and surjective over a component, we obtain the claim.
\end{proof}

\begin{claim}
$\gamma \in \ulMCorel (M,T(00))$.
\end{claim}

\begin{proof}
We make a preliminary reduction as follows:
since the assertion depends only on the isomorphism class of $T$ in $\ulMSm^\Sq$, we may assume that $T$ is ambient by Lemma \ref{lem:ambient-diagram}.
Moreover, since $T$ is a pull-back diagram, we have $T(00) \cong T(10) \times^c_{T(11)} T(01)$, where the right hand side is the canonical model of fiber product in Def. \ref{def:canonical-model}. 
Therefore, by replacing $T(00)$ with (the normalization of) $T(10) \ctimes_{T(11)} T(01)$ (this preserves the condition that $T$ is ambient by the construction of canonical model), we may assume that $\ol{q}_T^\ast T(10)^\infty$ and $\ol{v}_T^\ast T^\ast (01)$ have a universal supremum in the sense of \cite[Def. 1.10.2]{modsheaf1} and that $T(00)^\infty = \sup (\ol{q}_T^\ast T(10)^\infty , \ol{v}_T^\ast T(01)^\infty)$.

Let $\ol{\gamma}$ be the closure of $\gamma$ in $\ol{M} \times \ol{T}(00)$.
First we check that $\ol{\gamma}$ is proper over $\ol{M}$.
Note that the natural map $\ol{\gamma} \to \ol{M}$ factors as $\ol{\gamma} \to \ol{\alpha}_1 \times_{\ol{M}} \ol{\alpha}_2 \to \ol{M}$. 
The first map is proper since the natural map $\ol{T}(00) \to \ol{T}(10) \times_{\ol{T}(11)} \ol{T}(01)$ is proper by construction of the canonical model of fiber product, and the latter map is proper since $\ol{
\alpha}_i$ are proper over $\ol{M}$ by assumption.
This shows that $\ol{\gamma} \to \ol{M}$ is proper, as desired.

Next we check the modulus condition.
Let $\ol{\gamma}^N$ be the normalization of $\ol{\gamma}$. 
Similarly, let $\ol{\alpha}_1$ (resp. $\ol{\alpha}_2$) be the closure of $\alpha_1$ (resp. $\alpha_2$) in $\ol{M} \times \ol{T}(10)$ (resp. $\ol{M} \times \ol{T}(01)$), and $\ol{\alpha}_i^N$ the normalization of $\ol{\alpha}_i$. 
By assumption, we have $\alpha_1 \in \ulMCorel (M,T(10))$ and $\alpha_2 \in \ulMCorel (M,T(01))$, which means $M^\infty |_{\ol{\alpha}_1^N} \geq T(10)^\infty |_{\ol{\alpha}_1^N}$ and $M^\infty |_{\ol{\alpha}_2^N} \geq T (01)^\infty |_{\ol{\alpha}_2^N}$.
Since $\gamma \to \alpha_i$ are dominant for $i=1,2$, we obtain morphisms $\ol{\gamma}^N \to \ol{\alpha}_i^N$ by the universal property of normalization. 
Therefore, the above inequalities imply 
\[
M^\infty |_{\ol{\gamma}^N} \geq \ol{q}_T^\ast T(10)^\infty |_{\ol{\gamma}^N}, \ \ M^\infty |_{\ol{\gamma}^N} \geq \ol{v}_T^\ast T(01)^\infty |_{\ol{\gamma}^N}.
\]
Thus, since $\ol{q}_T^\ast T(10)^\infty$ and $\ol{v}_T^\ast T(01)^\infty$ have a universal supremum and since $T(00)^\infty = \sup (\ol{q}_T^\ast T(10)^\infty, \ol{v}_T^\ast T(01)^\infty)$ by assumption, we obtain  
\begin{align*}
M^\infty |_{\ol{\gamma}^N} 
&\geq \sup (\ol{q}_T^\ast T(10)^\infty |_{\ol{\gamma}^N} ,  \ol{v}_T^\ast T(01)^\infty |_{\ol{\gamma}^N}) \\
&= \sup (\ol{q}_T^\ast T(10)^\infty ,  \ol{v}_T^\ast T(01)^\infty)  |_{\ol{\gamma}^N} \\
&= T(00)^\infty |_{\ol{\gamma}^N} .
\end{align*}
by \cite[Remark 1.10.3 (3)]{modsheaf1}.
This finishes the proof of the claim. 
\end{proof}

By construction, we have $\alpha_1 = q_{T+} (\gamma)$ and $\alpha_2 = v_{T+} (\gamma)$.\
This finishes the proof of Proposition \ref{prop:fb-fc}.
\end{proof}

\section{Off-diagonal functor}\label{section-OD}
\setcounter{subsection}{1}

We introduce the ``off-diagonal'' functor, which is a key notion used in the definition of the cd-structure on $\MSm$.

\begin{defn}
Define $\ulMEt$ as a category such that 
\begin{enumerate}
\item objects are  those morphisms $f : M \to N$ in $\ulMSm$ such that $f^\o : M^\o \to N^\o$ is \'etale, and
\item morphisms of $f : M_1 \to N_1$ and $g : M_2 \to N_2$ are those pairs of morphisms $(s: M_1 \to M_2 , t:N_1 \to N_2)$ which are compatible with $f,g$ such that $s^\o$ and $t^\o$ are \emph{open immersions}.
\end{enumerate}

Define $\MEt$ as 
the full subcategory of $\ulMEt$ 
consisting of those $f : M \to N$ such that $M,N \in \MSm$.
\end{defn}

\begin{defn}
For modulus pairs $M$ and $N$, we define \emph{the disjoint union of $M$ and $N$} by 
\[
M \sqcup N := (\ol{M} \sqcup \ol{N}, M^\infty \sqcup N^\infty ).
\]
We have $(M \sqcup N)^\o = M^\o \sqcup N^\o$, and $M \sqcup N$ represents a coproduct of $M$ and $N$ in the category $\ulMSm$.
\end{defn}

\begin{thm}\label{thm:def-OD}
There exists a functor 
\[
\OD : \ulMEt \to \ulMSm
\]
such that for any $f :  M \to N$, there exists a functorial decomposition 
\[
M \times_N M \cong M \sqcup \OD(f).
\]
Moreover, we have $\OD (f)^\o = M^\o \times_{N^\o} M^\o \setminus \Delta (M^\o)$, where $\Delta : M^\o \to M^\o \times_{N^\o} M^\o$ is the diagonal morphism. 
In particular, if $f^\o$ is an open immersion, then $\OD (f)^\o = \emptyset$, hence $\OD (f) = \emptyset$.
Moreover, the functor $\OD$ restricts to a functor 
\[
\OD : \MEt \to \MSm.
\]
We call the functors \emph{the off-diagonal functors}.
\end{thm}

\begin{proof}
First, we prove that for any $f : M \to N$ in $\ulMEt$, there exists a morphism $i : X \to M \times_N M$ such that the induced morphism 
\[
M \sqcup X \xrightarrow{\Delta \sqcup i} M \times_N M
\]
is an isomorphism in $\ulMSm$.
Take any object $f : M \to N$ in $\ulMEt$.
Since $f^\o$ is \'etale and separated by the assumption, the diagonal morphism $\Delta : M^\o \to M^\o \times_{N^\o} M^\o$ is an open and closed immersion.
Therefore, we obtain a decomposition into two connected components:
\[
M^\o \times_{N^\o} M^\o = \Delta (M^\o) \sqcup (M^\o \times_{N^\o} M^\o - \Delta (M^\o)).
\]

Let $P$ denote the canonical model of fiber product $M \ctimes_N M$ as in Def. \ref{def:canonical-model}.
Note that $P^\o = M^\o \times_{N^\o} M^\o$.

Define a closed immersion $\ol{i}_{\Delta} : \ol{\Delta}(f) \to \ol{P}$ as the scheme-theoretic closure of the open immersion $\Delta (M^\o) \to P^\o \to \ol{P}$.
Set $\Delta (f)^\infty := \ol{i}_{\Delta}^\ast P^\infty$ and $\Delta (f):= (\ol{\Delta}(f), \Delta (f)^\infty)$.
Then $\ol{i}_\Delta$ induces a minimal morphism $i_\Delta : \Delta (f) \to P$. 
Moreover, we have $\Delta (f)^\o = \Delta (M^\o)$.

Similarly, define a closed immersion $\ol{i}_{\OD} : \ol{\OD (f)} \to \ol{P}$ as the scheme-theoretic closure of the open immersion $M^\o \times_{N^\o} M^\o - \Delta (M^\o) \to P^\o \to \ol{P}$.
Set $\OD(f)^\infty := \ol{i}_{\OD}^\ast P^\infty$ and $\OD (f) := (\ol{\OD (f)}, \OD (f)^\infty)$.
Then $\ol{i}_\OD$ induces a minimal morphism $i_\OD : \OD (f) \to P$.
Moreover, we have $\OD (f)^\o = M^\o \times_{N^\o} M^\o - \Delta (M^\o)$.

The morphisms $i_\Delta$ and $i_\OD$ induce a minimal morphism in $\ulMSm^\fin$:
\[
i_\Delta \sqcup i_\OD : \Delta (f) \sqcup \OD(f) \to P.
\]
By \eqref{def:mod-pair7} in Definition, \ref{def:mod-pair}, this morphism is an isomorphism in $\ulMSm$ (not in $\ulMSm^\fin$) since $(i_\Delta \sqcup i_\OD)^\o = i_\Delta^\o \sqcup i_\OD^\o : \Delta (f)^\o \sqcup \OD (f)^\o \to P^\o \cong M^\o \times_{N^\o} M^\o$ is an isomorphism in $\Sm$, and since $\ol{i}_\Delta \sqcup \ol{i}_\OD : \ol{\Delta (f)} \sqcup \ol{\OD (f)} \to \ol{P}$ is proper by construction. 

We claim $\Delta (f) \cong M$.
Let $\Delta : M \to P (\cong M \times_N M)$ be the diagonal morphism. 
Then, the composite  $M \xrightarrow{\Delta} P \cong \Delta (f) \sqcup \OD(f)$ factors through $\Delta (f)$.
The inverse morphism is given by $\Delta (f) \to P \xrightarrow{\mathrm{pr}_1} M$, where $\mathrm{pr}_1$ denotes the first projection $P \cong M \times_N M \to M$.

Thus, for any $f : M \to N$ in $\ulMEt$, we have obtained a decomposition 
\[
M \times_N M \cong M \sqcup \OD(f).
\]
Next we check the functoriality of $\OD(f)$.
Let $(f_1 : M_1 \to N_1) \to (f_2 : M_2 \to N_2)$ be a morphism in $\ulMEt$, i.e., a commutative diagram 
\[\xymatrix{
M_1 \ar[r]^{s} \ar[d]_{f_1} & M_2 \ar[d]^{f_2} \\
N_1 \ar[r]^{t} & N_2
}\]
where $f_1$, $f_2$, $s$ and $t$ are morphisms in $\ulMSm$ such that $f_1^\o$ and $f_2^\o$ are \'etale and $s^\o$ and $t^\o$ are open immersions.

We claim that there exists a unique morphism $\OD(f_1) \to \OD(f_2)$ such that the following diagram commutes:
\[\xymatrix{
M_1 \times_{N_1} M_1 \ar[r]  & M_2 \times_{N_2} M_2  \\
M_1 \sqcup \OD(f_1) \ar[r] \ar[u]^{\cong} & M_2 \sqcup \OD(f_2) \ar[u]_{\cong} .
}\]
The uniqueness is obvious by the commutativity of the above diagram.
For the existence, we need to show that the composite 
\[
\OD(f_1) \to  M_1 \times_{N_1} M_1 \to M_2 \times_{N_2} M_2 \cong M_2 \sqcup \OD(f_2)
\]
factors through $\OD(f_2)$.
To see this, it suffices to prove that the image of the morphism 
\[
M_1^\o \times_{N_1^\o} M_1^\o \setminus \Delta (M_1^\o) \to  M_1^\o \times_{N_1^\o} M_1^\o \xrightarrow{s^\o \times s^\o} M_2^\o \times_{N_2^\o} M_2^\o
\]
lands in $M_2^\o \times_{N_2^\o} M_2^\o \setminus \Delta (M_2^\o)$, which easily follows from the injectivity of the open immersion $s^\o$.
This finishes the proof.
\end{proof}

The off-diagonal functor is compatible with base change. 

\begin{prop}\label{prop:pullback-OD}
Let $f : M \to N$ be an object of $\ulMEt$, and $N' \to N$ any morphism in $\ulMSm$.
Then the base change $g:= f \times_N N'$ belongs to $\ulMEt$, and we have a natural isomorphism $\OD (g) \cong \OD (f) \times_N N'$.
\end{prop}

\begin{proof}
The first assertion holds since $g^\o = f^\o \times_{N^\o} N'^\o$ is \'etale as a base change of an \'etale morphism.
We prove the second assertion. 
Note $(M \times_N M) \times_{N} N' \cong M' \times_{N'} M'$, where $M' := M \times_{N} N'$. 
Consider the following diagram in $\ulMSm$:
\[\xymatrix{
(M \times_N M) \times_{N} N'   \ar[d] & \ar[l] (M \sqcup \OD (f)) \times_{N} N' & \ar[l] \ar[d]^{h} M' \sqcup (\OD (f) \times_N N') \\
M' \times_{N'} M' & & \ar[ll] M' \sqcup \OD (g)
}\]
where all the arrows, except for $h$, are natural isomorphisms in $\ulMSm$, and $h$ is defined to be the composite. 
By diagram chase, $h$ restricts to the identity map on $M'$ and an isomorphism $\OD (f) \times_N N' \to \OD (g)$. 
This finishes the proof.
\end{proof}

\section{The cd-structure}\label{section:cd-structure}

In this section, we introduce a cd-structure on $\MSm$, and prove its fundamental properties. 

\subsection{$\protect\ulMV$-squares}

First, we recall from \cite{modsheaf1} the cd-structure on $\ulMSm$.

\begin{defn}\label{def:ulMV}
\ 
\begin{enumerate}
\item An \emph{$\ulMVfin$-square} is a square $S \in (\ulMSm^\fin )^\Sq$ such that the morphisms in $S$ are minimal, and such that the resulting square
\[\xymatrix{
\ol{S}(00) \ar[r] \ar[d] & \ol{S}(01) \ar[d] \\
\ol{S}(10) \ar[r] & \ol{S}(11)
}\]
is an elementary Nisnevich square (on $\Sch$).

\item An $\ulMV$-square is a square $S \in \ulMSm^\Sq$ which belongs to the essential image of the inclusion functor $(\ulMSm^\fin )^\Sq \to \ulMSm^\Sq$.
\end{enumerate}
\end{defn}

\begin{prop}[{\protect\cite[Prop. 3.2.2]{modsheaf1}}]
The $\ulMV$-squares form a complete and regular cd-structure $P_{\ulMV}$ on $\ulMSm$. \qed
\end{prop}

\begin{defn}
The topology on $\ulMSm$ associated with the cd-structure $P_{\ulMV}$ is called \emph{the $\ulMV$-topology}.
\end{defn}

In the following, we describe $\OD$ for $\ulMV^\fin$ \& $\ulMV$-squares.

\begin{lemma}\label{lem:et-OD}
Let $f : U \to M$ be a minimal morphism such that $\ol{f} : \ol{U} \to \ol{M}$ is \'etale.
Then we have
\begin{align*}
&\ol{\OD (f)} = \ol{U} \times_{\ol{M}} \ol{U} - \Delta (\ol{U}) \\
&\OD (f)^\infty = \pi^\ast M^\infty \cap \ol{\OD (f)} ,
\end{align*}
where $\Delta : \ol{U} \to \ol{U} \times_{\ol{M}} \ol{U}$ is the diagonal, and $\pi : \ol{U} \times_{\ol{M}} \ol{U} \to \ol{M}$ is the natural morphism. 
\end{lemma}

\begin{proof}
Since $U^\o \times_{M^\o} U^\o - \Delta (U^\o)$ is dense in $\ol{U} \times_{\ol{M}} \ol{U} - \Delta (\ol{U})$ (as a complement of the divisor $U^\infty \times_{\ol{M}} \ol{U} \setminus \Delta (\ol{U})$), and since $U^\infty \times_{\ol{M}} \ol{U} = \ol{U} \times_{\ol{M}} U^\infty = \pi^\ast M^\infty$, the assertion follows from the construction of $\OD (f)$. This finishes the proof.
\end{proof}

\begin{prop}\label{prop:MV-OD}
Let $S$ be an $\ulMV^\fin$-square of the form 
\[\xymatrix{
S(00) \ar[r]^{v_S} \ar[d]_{q_S} & S(01) \ar[d]^{p_S} \\
S(10) \ar[r]^{u_S} & S(11).
}\]
Then the morphism $\OD (q_S) \to \OD (p_S)$ is an isomorphism in $\ulMSm^\fin$.
\end{prop}

\begin{proof}
Let $S$ be an $\ulMV^\fin$-square. 
Then, since $\ol{S}$ is an elementary Nisnevich square, we have a natural isomorphism
\[
\ol{S}(00) \times_{\ol{S}(10)} \ol{S}(00) - \Delta_0 (\ol{S}(00)) \xrightarrow{\sim} \ol{S}(01) \times_{\ol{S}(11)} \ol{S}(01) - \Delta_1 (\ol{S}(01)),
\]
where $\Delta_i : \ol{S}(0i) \to \ol{S}(0i) \times_{\ol{S}(1i)} \ol{S}(0i)$ is the diagonal for each $i=0,1$.
Then, in view of Lemma \ref{lem:et-OD}, the minimality of $u_S,p_S,q_S$ shows that the isomorphism as above induces an isomorphism $\OD (q_S) \to \OD (p_S)$ in $\ulMSm^\fin$.
This finishes the proof.
\end{proof}

\begin{cor}
Let $S$ be an $\ulMV$-square. Then the natural morphism $\OD (q_S) \to \OD (p_S)$ is an isomorphism in $\ulMSm$.
\end{cor}

\begin{proof}
By definition of $\ulMV$-square, there exists an $\ulMVfin$-square $S'$ which is isomorphic to $S$. Then, noting that there are natural isomorphisms $\OD (q_{S}) \cong \OD (q_{S'})$ and $\OD (p_{S}) \cong \OD (p_{S'})$ in $\ulMSm$, the assertion follows from Proposition \ref{prop:MV-OD}.
\end{proof}

\subsection{$\protect\MV$-squares}

\begin{defn}\label{def:new-MV}
Let $T$ be an object of $\MSm^\Sq$ of the form \eqref{eq:square-T}.
Then $T$ is called an \emph{$\MV$-square} if the following conditions hold:
\begin{enumerate}
\item \label{thm:new-MV1} $T$ is a pull-back square in $\MSm$.
\item \label{thm:new-MV2} There exist an $\ulMV$-square $S$ such that $S(11) \in \MSm$ and a morphism $S \to T$ in $\ulMSm^\Sq$ such that the induced morphism $S^\o \to T^\o$ is an isomorphism in $\Sm^\Sq$ and $S(11) \to T(11)$ is an isomorphism in $\MSm$. 
In particular, $T^\o$ is an elementary Nisnevich square. 
\item \label{thm:new-MV3} $\OD(q_T) \to \OD (p_T)$ is an isomorphism in $\MSm$.
\end{enumerate}

We let $P_{\MV}$ be the cd-structure on $\MSm$ consisting of $\MV$-squares. 
The topology on $\MSm$ associated with the cd-structure $P_{\MV}$ is called \emph{the $\MV$-topology} for short.
\end{defn}

\begin{remark}\label{rem:new-OD}
\ 

\begin{enumerate}
\item \label{rem:new-OD1} For any $T \in \MSm^\Sq$ with $T^\o$ an elementary Nisnevich square, the induced morphism $\OD (p_T)^\o \to \OD (q_T)^\o$ between interiors is an isomorphism in $\Sm$.
This follows easily from the definition of elementary Nisnevich squares.
\item \label{rem:new-OD2} If $p_T^\o$ and $q_T^\o$ are open immersions, then $\OD (q_T) = \OD (p_T) = \emptyset$. In particular, we have $\OD (q_T ) \cong \OD (p_T)$.
\end{enumerate}
\end{remark}

\begin{prop}\label{MV-bc}
Let $T$ be a square in $\MSm^\Sq$ which satisfies Condition \eqref{thm:new-MV1} (resp. \eqref{thm:new-MV2}, resp. \eqref{thm:new-MV3}).
Then, for any morphism $M \to T(11)$ in $\MSm$, the base change square $T_M := T \times_{T(11)} M$ satisfies  \eqref{thm:new-MV1} (resp. \eqref{thm:new-MV2}, resp. \eqref{thm:new-MV3}).
\end{prop}

\begin{proof}
Since base change of a pull-back diagram is a pull-back diagram, Condition \eqref{thm:new-MV1} is preserved by base change. 
Prop. \ref{prop:pullback-OD} shows that \eqref{thm:new-MV3} is preserved by the base change. 

Finally, we prove that Condition \eqref{thm:new-MV2} is preserved by base change. 
Let $S \to T$ be a morphism as in \eqref{thm:new-MV2}, and let $M \to T(11)$ be any morphism in $\MSm$. 
Then we obtain a morphism $S_M \to T_M$, where $S_M := S \times_{S(11)} M$ and $T_M := T \times_{T(11)} M$. 
Since $S(11) \cong T(11)$, we obtain $S_M (11) \cong T_M(11)$. 
Moreover, $S_M$ is an $\ulMV$-square as the base change of an $\ulMV$-square (see \cite[Theorem 4.1.2]{modsheaf1}), and we have $S_M^\o \cong T_M^\o$.
Therefore, the morphism $S_M \to T_M$ satisfies the requirement in \eqref{thm:new-MV2}.
This finishes the proof.
\end{proof}

\subsection{Completeness}

\begin{thm}\label{thm:completeness}
The cd-structure $P_{\MV}$ is complete. 
\end{thm}

\begin{proof}
By \cite[Lemma 2.5]{cdstructures}, it suffices to prove the following assertions:
\begin{enumerate}
\item Any morphism with values in $\emptyset = (\emptyset , \emptyset)$ is an isomorphism. 
\item For any $T \in P_{\MV}$ and for any $M \to T(11)$ in $\MSm$, the square $T_M := T \times_{T(11)} M$, which is obtained by base change, belongs to $P_{\MV}$.
\end{enumerate}
(1) is obvious, and (2) is a direct consequence of Proposition \ref{MV-bc}.
\end{proof}

\subsection{Regularity}

\begin{thm}\label{thm:regularity}
The cd-structure $P_{\MV}$ is regular. 
\end{thm}

\begin{proof}
By \cite[Lemma 2.11]{cdstructures}, it suffices to prove that for any $T \in P_{\MV}$, the following assertions hold:
\begin{enumerate}
\item $T$ is a pull-back square in $\MSm$.
\item $u_T : T(10) \to T(11)$ is a monomorphism.
\item The fiber products $T(01) \times_{T(11)} T(01)$ and $T(00) \times_{T(10)} T(00)$ exist in $\MSm$, and the derived square
\[\xymatrix{
T(00) \ar[r] \ar[d]_{\Delta_{q_T}} & T(01) \ar[d]^{\Delta_{p_T}} \\
T(00) \times_{T(10)} T(00) \ar[r] & T(01) \times_{T(11)} T(01)
}\]
which we denote by $d(T)$, belongs to $P_{\MV}$.
\end{enumerate}

(1) is by the definition of $\MV$-squares. 
(2) holds since $u_T^\o : T^\o (10) \to T^\o (11)$ is an open immersion. 
We prove (3): we check the conditions in Def. \ref{def:new-MV} for $d(T)$.

Since $\Delta_{p_T}^\o$ and $\Delta_{q_T}^\o$ are open immersions, we have $\OD (\Delta_{q_T}) \cong \emptyset \cong \OD (\Delta_{p_T})$ by Theorem \ref{thm:def-OD}. 
Hence $d(T)$ satisfies \eqref{thm:new-MV3} in Def. \ref{def:new-MV}.

Note that $d(T)$ is isomorphic in $\MSm^\Sq$ to the following diagram:
\[\xymatrix{
T(00) \ar[r] \ar[d]_{} & T(01) \ar[d]^{} \\
T(00) \sqcup \OD (q_T) \ar[r]^{} & T(01) \sqcup \OD (p_T)
}\]
where the vertical maps are the canonical inclusions, and the horizontal maps are induced by $v_T$.
It is easy to see that this diagram is a pull-back diagram, i.e., $d(T)$ satisfies \eqref{thm:new-MV1} in Def. \ref{def:new-MV}.
Indeed, suppose that we are given a pair of morphisms $f : M \to T(01)$ and $g : M \to T(00) \sqcup \OD (q_T)$ which coincide at $T(01) \sqcup \OD (p_T)$.
Then, one sees that $g^\o : M^\o \to T(00)^\o \sqcup \OD (q_T)^\o$ factors through $T(00)^\o$, which implies that $g$ factors through $T(00)$.

We are reduced to checking the condition \eqref{thm:new-MV2} for $d(T)$.
Consider the following diagram in $\ulMSm$:
\[\xymatrix{
(T(00)^\o,\emptyset) \ar[r] \ar[d]_{} & T(01) \ar[d]^{} \\
(T(00)^\o,\emptyset)  \sqcup \OD (q_T) \ar[r]^{} & T(01) \sqcup \OD (p_T)
}\]
which we denote by $d(T)_0$, 
where the vertical maps are the canonical inclusions. 
Then $d(T)_0$ is an $\ulMV$-square since $\OD (q_T) \cong \OD (p_T)$, and there exists a natural morphism $d(T)_0 \to d(T)$.
It induces an isomorphism $d(T)_0^\o \cong d(T)^\o$, and we have $d(T)_0(11) \cong d(T)(11)$.
Therefore, $d(T)$ satisfies \eqref{thm:new-MV2} in Def. \ref{def:new-MV}.
This finishes the proof.
\end{proof}

\begin{thm}\label{thm:sheaf-criterion}
Let $F$ be a presheaf with values in $\Sets$ on $\MSm$. 
Then $F$ is a sheaf with respect to the $\MV$-topology if and only if $F(\emptyset ) = 0$ and for any $\ulMV$-square $T \in P_{\MV}$, the square 
\[\xymatrix{
F(T(11)) \ar[r] \ar[d] & F(T(10)) \ar[d] \\ 
F(T(01)) \ar[r] & F(T(00))
}\]
is cartesian. 
\end{thm}

\begin{proof}
This follows from \cite[Corollary 2.17]{cdstructures}, Theorem \ref{thm:completeness} and Theorem \ref{thm:regularity}.
\end{proof}

\subsection{Subcanonicity}\label{subsection:subcanonicality}

In this subsection, we prove the following result.
Recall that a Grothendieck topology is \emph{subcanonical} if every representable presheaf is a sheaf. 

\begin{thm}\label{thm:subcanonicality}
The $\ulMV$-topology and the $\MV$-topology are subcanonical.
\end{thm}

We need the following elementary observation. 

\begin{lemma}\label{lem:subcan-cocart}
Let $P$ be a complete and regular cd-structure on a category $\sC$.
Then the topology associated with $P$ is subcanonical if and only if every square $T \in P$ is cocartesian in $\sC$.
\end{lemma}

\begin{proof}
Let $\sY$ denote the Yoneda embedding of $\sC$ into the category of presheaves on $\sC$.
All squares $T \in P$ are cocartesian in $\sC$ if and only if for any $T \in P$ and for any $X \in \sC$, the square 
\begin{equation}\label{eq:YT}\begin{gathered}\xymatrix{
\sY (X) (T(11)) \ar[r]^{u_T^\ast} \ar[d]_{p_T^\ast} & \sY (X) (T(10)) \ar[d]^{q_T^\ast} \\
\sY (X) (T(01)) \ar[r]^{v_T^\ast} & \sY (X) (T(00))
}\end{gathered}\end{equation}
is cartesian in $\sC$.
The latter condition is equivalent to that for any $X \in \sC$, the representable presheaf $\sY (X)$ is a sheaf for the topology associated with $P$ by \cite[Cor. 2.17]{cdstructures}.
This finishes the proof.
\end{proof}

We also need the following results:

\begin{lemma}[\protect{\cite[Lem. 2.2]{KP}}]\label{lKL} 
Let $f:X\to Y$ be a surjective morphism of normal integral schemes, 
and let $D,D'$ be two Cartier divisors on $Y$. 
If $f^*D'\le f^*D$, then $D'\le D$. \qed
\end{lemma}

\begin{prop}\label{prop:cocartesian}
\ 
\begin{enumerate}
\item Any $\ulMV$-square is cocartesian in $\ulMSm$.
\item Any $\MV$-square is cocartesian in in $\ulMSm$, hence in $\MSm$.
\end{enumerate}
\end{prop}

\begin{proof}
(1): Let $S$ be an $\ulMV$-square. 
We may assume that $S$ is an $\ulMVfin$-square since cocartesian-ness is stable under isomorphisms. 
Let $S(10) \to M$ and $S(01) \to M$ be morphisms in $\ulMSm$ which coincide after restricted to $S(00)$.
Since $S^\o$ is an elementary Nisnevich square, it is cocartesian in $\Sm$.
Therefore, the morphisms  $S(10)^\o \to M^\o$ and $S(01)^\o \to M^\o$ induce a unique morphism $h^\o : S(11)^\o \to M^\o$.
It suffices to check that $h^\o$ induces a morphism $S(11) \to M$ in $\ulMSm$.

Let $\Gamma$ be the graph of the rational map $\ol{S}(11) \dashrightarrow \ol{M}$, and let $\Gamma^N \to \Gamma$ be the normalization. 
For any $(ij) \in \Sq$, set
\[
S_1 (ij) := (\ol{S}(ij) \times_{\ol{S}(11)} \Gamma^N , S^\infty (ij) \times_{\ol{S}(11)} \Gamma^N).
\]
The minimal morphisms $S_1(ij) \to S_1(kl)$ are induced by $S(ij) \to S(kl)$ for all $(ij) \to (kl)$ in $\Sq$, and they form an $\ulMVfin$-square $S_1$.
Moreover, $S_1(ij)$ are normal for all $(ij) \in Sq$, and the composites 
\[
\ol{h}_{ij} : \ol{S}_1 (ij) \to \ol{S} (11) \dashrightarrow \ol{M}
\]
are morphisms of schemes for all $(ij) \in \Sq$ by construction. 
Moreover, the morphisms $\ol{S}_1(ij) \to \ol{S}(ij)$ are proper (by the properness of $\Gamma$ over $\ol{S}(11)$).
Therefore, by the minimality of $S_1(ij) \to S(ij)$, the morphism $S_1 \to S$ is an isomorphism in $\ulMSm^\Sq$.

\begin{claim}\label{claim:3.4.5}
$S_1^\infty (11) \geq \ol{h}_{11}^\ast M^\infty$.
\end{claim}

\begin{proof}
The admissibilities of $S(10) \to M$ and $S(01) \to M$ implies those of $S_1 (10) \to M$ and $S_1 (01) \to M$.
Since $\ol{S}_1(10)$ and $\ol{S}_1(01)$ are normal, we have 
\[
S_1 (ij)^\infty \geq \ol{h}_{ij}^\ast M^\infty 
\]
for $(ij)=(10), (01)$.
Since $\ol{S}_1(10) \sqcup \ol{S}_1 (01) \to \ol{S}_1 (11)$ is a surjection between normal schemes and since $S_1 (10) \to S_1(11)$ and $S_1(01) \to S_1(11)$ are minimal, Lemma \ref{lKL} implies 
\[
S_1 (11)^\infty \geq \ol{h}_{11}^\ast M^\infty .
\]
This finishes the proof.
\end{proof}

By Claim \ref{claim:3.4.5}, we have a morphism $S_1(11) \to M$ in $\ulMSm^\fin$.
The composite $S(11) \xleftarrow{\sim} S_1(11) \to M$ gives the desired morphism.
The uniqueness of the morphism follows from the fact that the elementary Nisnevich square $S^\o$ is cocartesian in $\Sm$.
This finishes the proof of (1).

(2): Let $T$ be an $\MV$-square. 
Then Condition (2) of Definition \ref{eq:square-T} shows that there exists an $\ulMV$-square $S$ and a morphism $S \to T$ in $\ulMSm^\Sq$ such that $S(11) \cong T(11)$.
Let $f : T(10) \to M$ and $g : T(01) \to M$ be morphisms in $\ulMSm$ which coincide after restricted to $T(00)$.
Then the composites 
\[
f_S : S(10) \to T(10) \to T(11) , \ \ g_S : S(01) \to T(01) \to T(11)
\]
coincide after restricted to $S(00)$.
Then $f_S$ and $g_S$ induce a unique morphism $h : T(11) \cong S(11) \to M$ since $S$ is cocartesian in $\ulMSm$ by (1).
Since $S^\o \cong T^\o$, we have $h \circ u_T = f$ and $h \circ p_T = g$.
This finishes the proof of Proposition \ref{prop:cocartesian}.
\end{proof}

\begin{proof}[Proof of Theorem \ref{thm:subcanonicality}]
This follows from Lemma \ref{lem:subcan-cocart} and Proposition \ref{prop:cocartesian} (1) and (2).
This finishes the proof.
\end{proof}

\section{Mayer-Vietoris sequence}\label{section:Mayer-Vietoris}

\subsection{Easy Mayer-Vietoris}

\begin{defn}\label{def:pre-yoneda}
For any sheaf $F$ on a site $\sC$, we denote by $\Z F$ the sheaf associated with the presheaf $\sC \ni X \mapsto \Z (F(X))$, where for any set $S$, we denote by $\Z S$ the free abelian group generated on $S$.

For any $M \in \ulMSm$ (resp. $\MSm$), we set $\Z (M) := \Z \sY (M)$, where $\sY (M)$ denotes the presheaf of sets represented by $M$.
\end{defn}

\begin{thm}\label{thm:easy-MV}
Let $T$ be an $\MV$-square. 
Then the complex of sheaves on $\MSm$
\[
0 \to \Z (T(00)) \to \Z (T(10)) \oplus \Z (T(01)) \to \Z (T(11)) \to 0
\]
is exact. 
\end{thm}

\begin{proof}
This follows from \cite[Lemma 2.18]{cdstructures}, Theorem \ref{thm:regularity} and Theorem \ref{thm:subcanonicality}.
\end{proof}

\subsection{Mayer-Vietoris with transfers}

\begin{thm}\label{thm:exactness}
Let $T \in \MSm^\Sq$.
Assume that $T^\o$ is an elementary Nisnevich square, and that $T$ satisfies Conditions \eqref{thm:new-MV1} and \eqref{thm:new-MV3} in Def. \ref{def:new-MV}.
Recall the following notation from Def. \ref{def:new-MV}:
\begin{equation}\label{eq:T} \vcenter{\xymatrix{
T(00) \ar[r]^{v_T} \ar[d]_{q_T} & T(01) \ar[d]^{p_T} \\
T(10) \ar[r]^{u_T} & T(11).
}} \end{equation}

Then, for any $M \in \ulMSm$, the following complex of abelian groups is exact:
\begin{multline}\label{eq:middle-exactness-tr}
0 \to \Z_\tr (T(00))(M) \xrightarrow{(q_{T_\ast} , v_{T\ast})}  \Z_\tr (T(10))(M) \oplus  \Z_\tr (T(01))(M) \\ \xrightarrow{p_{T\ast } - u_{T\ast}}  \Z_\tr (T(11))(M).
\end{multline}
\end{thm}

\begin{proof}
The assertion is equivalent to that the commutative square
\[\xymatrix{
\ulMCor (M,T(00)) \ar[r]^{v_{T\ast}} \ar[d]_{q_{T\ast}} & \ulMCor (M,T(01)) \ar[d]^{p_{T\ast}} \\
\ulMCor (M,T(10)) \ar[r]^{u_{T\ast}} & \ulMCor (M,T(11))
}\]
is cartesian. 
Note that the horizontal maps are injective. 

A key observation is the following lemma.
Recall the notation from Proposition \ref{prop:functor-elem}.

\begin{lemma}\label{lem:resurgence}
Let $\alpha_1, \alpha_2 \in \ulMCorel (M,T(01))$ be elementary correspondences with $\alpha_1 \neq \alpha_2$.
Assume that $p_{T+} (\alpha_1) = p_{T+} (\alpha_2)$ holds in  $\ulMCorel (M,T(11))$.
Then $\alpha_1$ and $\alpha_2$ belong to the image of $v_{T\ast}$.
\end{lemma}

\begin{proof}
Set $P := T(01) \times_{T(11)} T(01)$, and consider the commutative diagram
\[\xymatrix{
\ulMCorel (M, P) \ar[r]^{\pr_{1+}} \ar[d]^{\pr_{2+}} & \ulMCorel (M, T(01)) \ar[d] \\
\ulMCorel (M, T(01)) \ar[r] & \ulMCorel (M, T(11))
}\]
in $\Set$.
By Proposition \ref{prop:fb-fc}, there exists an element $\gamma \in \ulMCorel (M,P)$ such that $\pr_{1+} (\gamma) = \alpha_1$ and $\pr_{2+} (\gamma) = \alpha_2$.

We have a canonical identification \[\ulMCorel (M,P) \cong \ulMCorel (M,T(01)) \sqcup \ulMCorel (M,\OD (p_T))\] induced by $P \cong T(01) \sqcup \OD (p_T)$.
Through this identification, we regard $\ulMCorel (M,\OD (p_T ))$ as a subset of $\ulMCorel (M,P)$.

\begin{claim}
$\gamma \in \ulMCorel (M,\OD (p_T ))$.
\end{claim}

\begin{proof}
Let $\xi_1, \xi_2$ and $\zeta$ be the generic points of $\alpha_1, \alpha_2$ and $\gamma$.
Then $\zeta$ lies over $\xi_1$ and $\xi_2$.
Since $\xi_1 \neq \xi_2$ by the assumption that $\alpha_1 \neq \alpha_2$, we have $\zeta \notin M^\o \times \Delta (T(01)^\o)$, where $\Delta (T(01)^\o)$ denotes the image of $\Delta : T(01)^\o \to T(01)^\o \times_{T(11)^\o} T(01)^\o$.
This implies that $\zeta \in M^\o \times \OD (p_T)^\o$.
Therefore, we have 
\[
\gamma \in \Cor (M^\o , \OD (p_T)^\o ) \cap \ulMCorel (M,P) = \ulMCorel (M,\OD (p_T)).
\]
This finishes the proof of the claim.
\end{proof}

By construction, we have $\alpha_i = \pr_{i} (\gamma) = |(\pr_{i})_{\ast} (\gamma)|$, where $\pr_i : T(01)^\o \times_{T(11)^\o} T(01)^\o \to T(01)^\o$, $i=1,2$, are the projections. 
Therefore, in order to prove $\alpha_i \in \ulMCor (M,T(00))$ for $i=1,2$, it suffices to prove that $\gamma \in \ulMCor (M,T(00) \times_{T(10)} T(00))$.
Since $\gamma \in \ulMCor (M,\OD (p_T))$ by the above claim, and since $\OD (q_T) \cong \OD(p_T)$ by Condition \eqref{thm:new-MV3} of Definition \ref{def:new-MV}, we have $\gamma \in \ulMCor (M,\OD(q_T)) \subset \ulMCor (M,T(00) \times_{T(10)} T(00))$. 
This finishes the proof of Lemma \ref{lem:resurgence}.
\end{proof}

Now we are ready to prove that the above square is cartesian.
Let $\alpha \in \ulMCor (M,T(01))$ and assume $p_{T\ast} (\alpha) \in \ulMCor (M,T(10))$. 
Write $\alpha = \sum_{i \in I} m_i \alpha_i$, where $I$ is a finite set, $m_i \in \Z - \{0\}$ and $\alpha_i$  are elementary correspondences which are distinct from each other. 
Then we have $\alpha_i \in \ulMCor (M,T(01))$ for all $i \in I$. 
Set 
\[
J := \{i \in I \ | \ \exists j \in I - \{i\}, \ |p_{T\ast} (\alpha_i)| = |p_{T\ast} (\alpha_{j})| \}.
\]
Then, by Lemma \ref{lem:resurgence}, we have $\alpha_i \in \ulMCor (M,T(00))$ for all $i \in J$. 
Let $i \in I - J$, and set $\beta := |p_{T\ast} (\alpha_i)|$. 
Then the coefficient of $\beta$ in $p_{T\ast} (\alpha)$ is non-zero, and therefore $\beta \in \ulMCorel (M,T(10))$. 
By Proposition \ref{prop:fb-fc}, there exists a unique element $\gamma \in \ulMCorel (M,T(00))$ such that $v_{T+} (\gamma) = \alpha_i$ and $q_{T+} (\gamma) = \beta$. 
Since $T(00)^\o \to T(01)^\o$ is an open immersion, this implies that $\alpha_i = \gamma \in \ulMCorel (M,T(00))$. 
This finishes the proof of the exactness of \eqref{eq:middle-exactness-tr}.
\end{proof}

Recall from \cite[Th. 2 (2)]{modsheaf1} that for any $M \in \ulMSm$, the presheaf $\Z_\tr (M)$ on $\ulMSm$ is a sheaf for the $\ulMV$-topology.

\begin{corollary}\label{cor:Mayer-Vietoris}
Let $T$ be an $\MV$-square. 
Then the following complex of sheaves on $\ulMSm$ for the $\ulMV$-topology is exact:
\begin{multline}\label{eq:exactness-tr}
0 \to \Z_\tr (T(00))\xrightarrow{(q_{T_\ast} , v_{T\ast})}  \Z_\tr (T(10))\oplus  \Z_\tr (T(01))\\ \xrightarrow{p_{T\ast } - u_{T\ast}}  \Z_\tr (T(11)) \to 0.
\end{multline}
\end{corollary}

\begin{proof}
By Theorem \ref{thm:exactness}, it suffices to prove the surjectivity of the last maps of the complexes. 
Take a morphism $S \to T$ in $\ulMSm^\Sq$ as in \eqref{thm:new-MV2} of Definition \ref{def:new-MV}.
Then the map
\[
\Z_\tr (S(10)) \oplus \Z_\tr (S(01)) \to \Z_\tr (S(11)) = \Z_\tr (T(11)) 
\]
is epi in $\ulMNST$ by \cite[Theorem 4.5.7]{modsheaf1}.
Since the map factors through 
\[
\Z_\tr (T(10)) \oplus \Z_\tr (T(01)),\] we are done. 
\end{proof}

\end{document}